\documentclass{amsart}
\usepackage{graphicx}
\usepackage{amsthm}
\usepackage{amsfonts}
\usepackage{amsmath}
\usepackage{amssymb}
\usepackage{tikz-cd}
\usepackage{mathtools}
\usepackage{mathrsfs}
\definecolor{darkgreen}{rgb}{0,0.50,0} 
\definecolor{darkred}{rgb}{0.55,0,0}
\definecolor{darkblue}{rgb}{0,0,0.6} 
\usepackage[pagebackref,colorlinks,citecolor=darkgreen,linkcolor=darkred,urlcolor=darkblue]{hyperref}
\usepackage[capitalise]{cleveref}

\title{Cutting and pasting pairs of manifolds with tangential structures}
\author{R.A. Vlierhuis}
\address{Vrije Universiteit Amsterdam}
\email{r.a.vlierhuis@student.vu.nl}
\date{June 2025}
\newtheorem{thmx}{Theorem}

\newtheorem{thm}{Theorem}[section]
\newtheorem*{thm*}{Theorem}

\newtheorem{lem}[thm]{Lemma}
\newtheorem{prop}[thm]{Proposition}
\theoremstyle{definition}
\newtheorem{ex}[thm]{Example}
\newtheorem{defi}[thm]{Definition}

\theoremstyle{remark}
\newtheorem*{rmk}{Remark}
\newcommand{\N}{\mathbb{N}}
\newcommand{\R}{\mathbb{R}}
\newcommand{\Z}{\mathbb{Z}}

\newcommand{\de}{\partial}
\newcommand{\im}{\operatorname{im}}
\renewcommand{\int}{\operatorname{int}}
\newcommand{\SK}{\operatorname{SK}}

\newcommand{\RP}{\mathbb{R}P}

\newcommand{\bSK}{\overline{\SK}}
\newcommand{\id}{\operatorname{id}}

\begin{document}

\begin{abstract}
    This paper studies cutting and pasting groups (SK-groups) of pairs of manifolds. By a pair of manifolds we mean a manifold with a submanifold of strictly smaller dimension. Existing results in the unoriented category by Komiya are generalized to manifolds with certain tangential structures. In this way multiple new splitting results for SK of pairs are obtained, in particular for SK of pairs of manifolds with a map to a reference space. We also prove that SK of manifolds with a map into $X\times Y$ with $Y$ simply connected is the same as SK of manifolds with map into $X$. This generalizes the result of Neumann that SK of manifolds with a map into a simply connected space is trivial.
\end{abstract}
\maketitle
\section{Introduction}
$\SK$-groups were introduced by Karras, Kreck, Neumann and Ossa in  their 1973 book \cite{skbook}. Given a manifold $M$, choosing a codimension one splitting submanifold with trivial normal bundle yields a splitting $M = M_1 \cup M_2$, where $M_1$ and $M_2$ are manifolds with boundary. If one now takes a diffeomorphism $\varphi: \de M_1 \cong \de M_2$, the pieces can be glued back along $\varphi$ to get a manifold potentially different from $M$. The $\SK$-groups measure how many different classes of manifolds there are under this cutting and pasting relation. Recently, there has been renewed interest in these SK-groups and their generalizations, see for example \cite{HOEKZEMA2022108105} and the preprints \cite{hoekzema2025} and \cite{merling2025}.

The present paper aims to extend cutting and pasting of pairs. A pair of manifolds is a manifold with a submanifold of strictly smaller dimension. This generalization of SK was considered first by Komiya \cite{komiya1986}, who proved in the unoriented case that cutting and pasting of pairs does not introduce any new obstructions. Informally, ``cutting and pasting a pair is the same as cutting and pasting the big manifold and the small manifold individually''. We shall extend his results from unoriented manifolds to  manifolds endowed with a sufficiently nice tangential $B$-structure. Our statement is the following.
\begin{thmx}[\cref{main-thm}]
Let $B\to BO_k$ be a strongly multiplicative structure, with $k\geq m > n$. There is a split short exact sequence
    \[0\to \SK^B_m \xrightarrow{i} \SK^B_{m,n} \xrightarrow{j} \SK^{B}_n(B_{m-n}) \to 0,\]
    with $i([M]) = [M,\emptyset]$ and $j([M,N]) = [N,\nu]$ where $\nu: N \to B_{m-n}$ is the $B$-structure for the normal bundle of $N$ in $M$.
\end{thmx}
This result enables many new splitting results. For example, we present one for unoriented SK with a map to a reference space $X$ (\cref{komiya-unori}), and an analogous result in the oriented case (\cref{komiya-ori}). These results may be phrased as saying that ``cutting and pasting a pair of maps is the same as cutting and pasting the maps individually''. While proving the statement in the oriented case we will also show the following, which generalizes Theorem 1 from \cite{neumann1975}.
\begin{thmx}[\cref{forget-Y}]
    Let $Y$ be a simply connected pointed space, and suppose $\pi_1(X)$ is finitely presented. Then the projection $p_*:\SK_n(X\times Y)\to\SK_n(X)$ is an isomorphism for all $n$.
\end{thmx}

\subsection*{Conventions}
In this text, a manifold is smooth and closed, not necessarily oriented. The reference space $X$ is assumed to be pointed and path connected. A star $*$ will denote the appropriate constant map.

\subsection*{Acknowledgments}
This work is a continuation of the author's BSc thesis, written under the supervision of Renee Hoekzema. I would like to thank her for inspiring me to write this paper and for making the project possible.

\section{Tangential structures}
We start by defining what we mean by a tangential structure on a bundle.

\begin{defi}[$B$-structure]
    Let $B$ be a pointed space, $B\to BO_n$ a pointed fibration. If $\xi_n:X\to BO_n$ is the classifying map for some vector/principal/fiber bundle over $X$, a $B$-structure on $\xi_n$ is a lift
    \[\begin{tikzcd}
                                      & B \arrow[d] \\
X \arrow[ru, dashed] \arrow[r, "\xi_n"] & BO_n       .
\end{tikzcd}\]
We call the data of $B$ together with the fibration the structure $B$ (when the fibration is clear from the context).
\end{defi}
By pullback, we can use the fibration $B\to BO_n$ to define what a $B$-structure is for any dimension $k\leq n$.
\begin{defi}
    If $B$ is as above, consider for all $k\leq n$ the fibrations $B_k\to BO_k$ defined by the pullbacks
    \[\begin{tikzcd}
B_k \arrow[r] \arrow[d] \arrow[dr, phantom, "\lrcorner", very near start] & B \arrow[d] \\
BO_k \arrow[r]          & BO_n    .   
\end{tikzcd}\]
If $\xi_k:X\to BO_k$ is the classifying map for some vector/principal/fiber bundle over $X$, a $B$-structure on $\xi_k$ is a lift
    \[\begin{tikzcd}
                                      & B_k \arrow[d] \\
X \arrow[ru, dashed] \arrow[r, "\xi_k"] & BO_k   .    
\end{tikzcd}\]
\end{defi}

We call a manifold $M$ with a $B$-structure on its tangent bundle a $B$-manifold. We need to be able to say when two $B$-manifolds are the same, generalizing the notion of an orientation-preserving diffeomorphism. This will be done as follows.

\begin{defi}[$B$-isomorphism]
    Let $E\to X$ and $E'\to X$ be two $B$-bundles. A $B$-isomorphism between $E$ and $E'$ is a homotopy between the classifying maps and a homotopy between the $B$-structures such that the triangle
    \[\begin{tikzcd}
                               & B_k \arrow[d] \\
X\times I \arrow[ru] \arrow[r] & BO_k         
\end{tikzcd}\]
    commutes. Note that if $E=E'$, this is just a vertical homotopy of the $B$-structures. If two $B$-structures are vertically homotopic, we say they are \textit{equivalent} $B$-structures.
\end{defi}

\begin{defi}[$B$-diffeomorphism]
    Let $M_1$ and $M_2$ be $B$-manifolds. A diffeomorphism $f:M_1\to M_2$ gives rise to an isomorphism of tangent bundles $df:TM_1\to TM_2$, which can be seen as a homotopy filling the triangle
    \[\begin{tikzcd}
M_1 \arrow[r, "f"] \arrow[rd] & M_2 \arrow[d] \\
                              & BO_k   .      
\end{tikzcd}\]
    A $B$-diffeomorphism from $M_1$ to $M_2$ is a diffeomorphism $f:M_1\to M_2$ together with a homotopy $H$ filling the triangle
    \[\begin{tikzcd}
M_1 \arrow[r, "f"] \arrow[rd] & M_2 \arrow[d] \\
                              & B_k    ,      
\end{tikzcd}\]
    so that $df$ and $H$ form a $B$-isomorphism of bundles over $M_1$.
\end{defi}

It is easy to check that the inverse of a $B$-isomorphism is again a $B$-isomorphism and that the inverse of a $B$-diffeomorphism is again a $B$-diffeomorphism.

Our results about cutting and pasting of pairs will apply to structures that satisfy a two-out-of-three theorem. For example, if $E\to M$ is a vector bundle, then two of the following being true implies the third: 
\begin{enumerate}
    \item $M$ is orientable
    \item $E$ is orientable (as a manifold)
    \item The bundle $E\to M$ is orientable.
\end{enumerate}
We formalize this for a general tangential structure as follows.


\begin{defi}[Strongly multiplicative]
   A structure $B\to BO_n$ is called strongly multiplicative if and only if for every $r,s\in\N$ with $r+s\leq n$ there exists a map $\psi_{r,s}:B_r\times B_s\to B_{r+s}$ such that
   \[\begin{tikzcd}
B_r\times B_s \arrow[r, "{\id\times\psi_{r,s}}"] \arrow[d] \arrow[dr, phantom, "\lrcorner", very near start] & B_r\times B_{r+s} \arrow[d] \\
BO_r\times BO_s \arrow[r, "{\id\times\phi_{r,s}}"]         & BO_r\times BO_{r+s}        
\end{tikzcd}\]
is a pullback square, where the vertical maps are the obvious ones and $\phi_{r,s}:BO_r\times BO_s\to BO_{r+s}$ is the natural embedding.
\end{defi}

This condition means that we choose a way to combine $B$-structures on $\xi^r$ and $\eta^s$ to a $B$-structure on $\xi^r\oplus\eta^s$, ensuring that $B$-structures on $\eta$ and $\xi\oplus\eta$ give rise to a \textit{unique} choice of $B$-structure on $\xi$.

\begin{rmk}
    The terminology `strongly multiplicative' is meant to reflect that we are demanding this two-out-of-three rule, and not just the weaker statement that a product of two $B$-bundles again has a $B$-structure. 
\end{rmk}

\begin{ex}
\label{example-structures}
    \begin{enumerate}
        \item The $k$-connected covers of $BO_n$ define strongly multiplicative structures. Indeed, one can check this by using the long exact sequence for the homotopy groups in a pullback square of fibrations. In particular, orientations and spin structures form strongly multiplicative structures, which can also be proven in a concrete manner.
        \item Another example is given by the $k$-orientability of Hoekzema \cite{Hoekzema_2018}. We call a manifold $k$-orientable if its Stiefel-Whitney classes up to dimension $2^k$ vanish. It is easy to show that the corresponding structure satisfies the strongly multiplicative property.
        \item If $B\to BO_n$ is any strongly multiplicative structure and $X$ is a space (we will assume it to be pointed and path connected), then another strongly multiplicative structure is given by the map $B\times X\to BO_n$ which first projects onto $B$ and then applies the initial fibration. Note that all pullbacks are of the form $B_m\times X\to BO_m$. That is, this structure represents a map from the $B$-manifold into $X$, with no further restrictions.
    \end{enumerate}
\end{ex}

\begin{prop}
    Let $\xi:E\to M$ be a vector bundle and $B$ a strongly multiplicative structure. If two of the following conditions hold, so does the third.
    \begin{enumerate}
        \item $M$ is a $B$-manifold
        \item $E$ is a $B$-manifold
        \item $\xi$ is a $B$-bundle.
    \end{enumerate}
\end{prop}
\begin{proof}
    Choose a connection on $\xi$. This is equivalent to a choice of splitting $TE = V\oplus H$ where $V = \ker d\xi$ and $H\cong \xi^*TM$. Additionally we can identify $V$ with $\xi^*E$. Suppose (1) and (2) hold. Then $TM$ is a $B$-bundle, hence so is $H$, thus by strong multiplicativity so is $V$. Pulling back $V$ along the zero section $M\to E$ we recover $\xi:E\to M$, which is thus a $B$-bundle.

    Now if (2) and (3) hold, $TE$ and $V$ are $B$-bundles, hence so is $H$ by strong multiplicativity. Again we pull back along the zero section, now to obtain $TM$. Finally, if (1) and (3) hold, then $V$ and $H$ are $B$-bundles, hence so is $TE$.
\end{proof}

\begin{prop}
    \label{sphere-bundles}
    Let $E$ be a vector bundle whose total space is a $B$-manifold, with $B$ strongly multiplicative. Then the total space of the sphere bundle $S(E)$ is a $B$-manifold.
\end{prop}
\begin{proof}
    This will follow at once from the fact that the normal bundle of $S(E)$ inside $E$ is trivial. Indeed, this normal bundle is clearly a line bundle, and a nonvanishing global section can be found by considering at every point $(x,v)$ in the sphere bundle the tangent vector given by the straight line to $(x,0)$.
\end{proof}

\begin{defi}[Pairs of manifolds]
    A pair of $B$-manifolds is a $B$-manifold $M$ with a submanifold $N$ of strictly smaller dimension that is also a $B$-manifold.
\end{defi}

\begin{prop}
    Let $M$ and $N$ be a pair of $B$-manifolds. Then the normal bundle of $N$ in $M$, denoted $\nu(N,M)$, is a $B$-bundle.
\end{prop}
\begin{proof}
    Let $\iota:N\to M$ be the inclusion and consider $\iota^*TM$, a bundle over $N$. There is a natural splitting $\iota^*TM = TN\oplus \nu(N,M)$ and so $\nu(N,M)$ has a $B$-structure by strong multiplicativity.
\end{proof}

We shall apply \cref{sphere-bundles} to the bundle $S:=S(\nu(N,M)\oplus \underline{\R})$. We will have that there is a tubular neighborhood of $N$ in $M$ that is $B$-diffeomorphic to a tubular neighborhood of $N$ in $S$, since the $B$-structures are both induced from the ones on $N$ and $M$.

We are now ready to define $\SK$ of pairs of $B$-manifolds following Komiya \cite{komiya1986}. Note that a $B$-diffeomorphism is defined exactly so that we can glue two $B$-manifolds with boundary along a $B$-diffeomorphism to get another $B$-manifold. This ensures that the following constructions are well-defined.

\begin{defi}
    Consider the monoid of $n$-dimensional manifolds with $B$-structure and operation disjoint union, denoted $\mathcal{M}_n^B$. We have an equivalence relation generated by $M_1 \cup_\phi M_2 \sim M_1 \cup_\psi M_2$, where $\phi,\psi:\de M_1\to \de M_2$ are $B$-diffeomorphisms. We get a cancellative semigroup given by $\mathcal{M}_n^B$ modulo this relation. Its Grothendieck group is the $\SK$-group with $B$-structure, $\SK_n^B$. 
\end{defi}

\begin{defi}
    Consider the monoid of $(m,n)$-pairs of $B$-manifolds, denoted $\mathcal{M}_{m,n}^B$. We have an equivalence relation generated by $(M_1,N_1) \cup_\phi (M_2,N_2) \sim (M_1,N_1) \cup_\psi (M_2,N_2)$, where $(M_i,N_i)$ are pairs of manifolds with boundary and $\phi,\psi: \de M_1\to \de M_2$ are $B$-diffeomorphisms such that $\phi|_{\de N_1}$ and $\psi|_{\de N_1}$ are $B$-diffeomorphisms onto $\de N_2$. We get a cancellative semigroup given by $\mathcal{M}_{m,n}^B$ modulo this relation and call its Grothendieck group $\SK_{m,n}^B$.
\end{defi}

\begin{rmk}
    We will adopt a few shorthand notations for the $\SK$-groups that will be discussed in more detail in the rest of the paper. Unoriented $\SK$ is given by the structure that is the identity on $BO_n$ and will be denoted $\SK_n^O$. Oriented $\SK$ is given by the 1-connected cover $BSO_n\to BO_n$ and will be denoted simply $\SK_n$. $\SK$ of $B$-manifolds with a map to $X$, given by the structure in \cref{example-structures}(3), will be denoted $\SK_n^B(X)$. We will use similar conventions for $\SK$ of pairs.
\end{rmk}
\section{Splitting results for SK of pairs}\label{results}

Our main theorem is the following.

\begin{thm}
\label{main-thm}
    Let $B\to BO_k$ be a strongly multiplicative structure, with $k\geq m > n$. There is a split short exact sequence
    \[0\to \SK^B_m \xrightarrow{i} \SK^B_{m,n} \xrightarrow{j} \SK^{B}_n(B_{m-n}) \to 0,\]
    with $i([M]) = [M,\emptyset]$ and $j([M,N]) = [N,\nu]$ where $\nu: N \to B_{m-n}$ is the $B$-structure for the normal bundle of $N$ in $M$.
\end{thm}

\begin{proof}
    Now that we've set up the language of strongly multiplicative $B$-structures, the proof is essentially the same as Komiya's proof in the unoriented case. Firstly, note that $j$ is well-defined because the normal bundle of $N$ in $M$ has a canonical $B$-structure by strong multiplicativity.

    It is clear that $i$ is injective and that $j\circ i = 0$. To see that $j$ is surjective, let $[N,\nu]\in\SK_n^B(B_{m-n})$. Let $E$ be the pullback of the universal bundle over $B_{m-n}$ by $\nu$ and consider the sphere bundle $S:=S(E\oplus\underline{\R})$. $E$ is a $B$-bundle by definition, so by \cref{sphere-bundles} $S$ is a $B$-manifold. We have a projection $\pi:S\to N$ and additionally we can view $N$ as a submanifold of $S$ by considering the north poles of all the spheres. The normal bundle of $N$ in $S$ is now clearly classified by $\nu$, so $j[S,N] = [N,\nu]$.
    
    The sequence splits since we can define $k:\SK_{m,n}(X)\to \SK_m(X)$ mapping $[M,N] \mapsto [M]$. It just remains to show that $\ker j \subset \im i$.
    
    To show that $\ker j \subset \im i$, let $[M,N]\in\ker j$. We can write $[M,N] = [M_1,N_1] - [M_2, N_2]$ in $\SK_{m,n}^B$. Since $[M,N]$ was in the kernel of $j$, it follows that $[N_1,\nu_1] = [N_2,\nu_2]$ in $\SK_n^B(B_{m-n})$, where $\nu_i:N_i \to B_{m-n}$ are the $B$-structures on the normal bundles of $N_i$ in $M_i$ for $i=1,2$. We can define $B$-manifolds $S_i$ for $i=1,2$ analogously to $S$ above.

    Let us find a closed neighborhood $T_i'$ of $N_i$ in $S_i$ that is $B$-diffeomorphic to a neighborhood $T_i$ of $N_i$ in $M_i$, using the tubular neighborhood theorem. We denote the interiors of the $T_i$ and $T_i'$ by $\mathring{T_i'}$ and $\mathring{T_i}$, respectively.

    We can split the pair $(M_i, N_i)$ into $(M_i-\mathring{T_i},\emptyset) \cup_{\text{id}} \overline{(T_i, N_i)}$. Make the following definitions:
    \begin{align*}
        K_i &:= (S_i - \mathring{T_i'}) \cup \overline{(M_i-\mathring{T_i})} \\
        K_i' &:= (M_i-\mathring{T_i}) \cup \overline{(M_i-\mathring{T_i})}.
    \end{align*}
    In $\SK_{m,n}^B$ we can now write
    \begin{align*}
        [M_i, N_i] + [K_i, \emptyset]
        =& [(M_i-\mathring{T_i},\emptyset)\cup \overline{(T_i,N_i)}] + [(S_i-\mathring{T_i'},\emptyset)\cup \overline{(M_i-\mathring{T_i},\emptyset)}]\\
        =& [(M_i-\mathring{T_i},\emptyset)\cup\overline{(M_i-\mathring{T_i},\emptyset)}] + [(S_i-\mathring{T_i'},\emptyset)\cup\overline{(T_i,N_i)}]\\
        =& [K_i',\emptyset] + [S_i, N_i].
    \end{align*}
    \noindent
    However, notice that $[S_1, N_1] = [S_2, N_2]$, so now computing $[M_1, N_1] - [M_2, N_2]$ will yield
        \[[M_1, N_1] - [M_2, N_2]
        = [K_1', \emptyset] + [K_2, \emptyset] - [K_1,\emptyset] - [K_2', \emptyset],\]
    proving that $[M_1, N_1] - [M_2, N_2]\in \im(i)$.
\end{proof}

In \cite{komiya1986}, the theorem is then strengthened by simplifying the right term in the sequence. There it is used that the natural map $\SK_n^O(BO(k)) \to \SK_n^O$ is an isomorphism. Such vanishing results do not hold in the generality of \cref{main-thm}. However, there is a few interesting cases where results can be proved, and the rest of this paper will be dedicated to this theory. In particular we will show the following two theorems.

\begin{thm}
    \label{komiya-unori}
    Let $X$ be a pointed path connected space and $n<m$. There is a split short exact sequence
    \[0 \longrightarrow \SK_m^O(X) \xrightarrow{i}  \SK_{m,n}^O(X) \xrightarrow{j}  \SK_n^O(X\times X) \longrightarrow 0,\]
    with $i([M,f]) = [(M,\emptyset), (f, \emptyset)]$ and $j([(M,N), (f,g)] = [N, f|_N\times g]$.
\end{thm}
\begin{proof}
    First apply \cref{main-thm} with the structure $BO\times X \to BO$ given by projection onto the first coordinate. The natural projection $\SK_n^O(X\times BO(m-n) \times X) \to \SK_n^O(X\times X)$ is an isomorphism, which will be proved in \cref{section-unori}.
\end{proof}

\begin{thm}
    \label{komiya-ori}
    Let $X$ be a pointed path connected space and $n<m$. There is a split short exact sequence
    \[0 \longrightarrow \SK_m(X) \xrightarrow{i} \SK_{m,n}(X) \xrightarrow{j}  \SK_n(X\times X) \longrightarrow 0,\]
    with $i([M,f]) = [(M,\emptyset), (f, \emptyset)]$ and $j([(M,N), (f,g)] = [N, f|_N\times g]$.
\end{thm}
\begin{proof}
    First apply \cref{main-thm} with the structure $BSO\times X \to BO$ given by projection onto the first coordinate and applying the natural map $BSO\to BO$. The natural projection $\SK_n(X\times BO(m-n) \times X) \to \SK_n(X\times X)$ is an isomorphism, which will be proved in \cref{section-ori}.
\end{proof}

\section{Some theorems about SK}

In this section we will state a number of theorems about SK that will be used in the remainder of the paper. More details and proofs can be found in \cite{skbook} or in \cite{neumann1975}. We start with two essential theorems for computing $\SK(X)$. First a definition.

\begin{defi}
    The group $\bSK_n(X)$ is the group $\SK_n(X)$ modulo bordism, that is, we set every class with a representative $(M,f)$ that bounds $(W,\Tilde{f})$ equal to zero.
\end{defi}

\begin{thm}[{\cite[Theorem 1.1]{skbook}}]
\label{skbook1.1}
    There is a split short exact sequence
    \[\begin{tikzcd}
    0 \arrow[r] & I_n(X) \arrow[r] & \SK_n(X) \arrow[r] & \bSK_n(X) \arrow[r] & 0
    \end{tikzcd},\]
    where $I_n(X)$ denotes the subgroup of $\SK_n(X)$ generated by $[S^n, *]$. This group is isomorphic to $\Z$ if $n$ is even and trivial if $n$ is odd. In particular it does not depend on $X$.
\end{thm}

This means we can compute $\SK_*(X)$ by only looking at $\bSK_*(X)$. The next theorem gives us a means of computing $\bSK_*(X)$. Let $F_n(X)$ be the subgroup of the bordism group $\Omega_n(X)$ given by all elements with a representative $(M,f)$ such that $M$ fibers over $S^1.$

\begin{thm}[{\cite[Theorem 1.2]{skbook}}]
\label{skbook1.2}
    There is a short exact sequence
    \[\begin{tikzcd}
    0 \arrow[r] & F_n(X) \arrow[r] & \Omega_n(X) \arrow[r] & \bSK_n(X) \arrow[r] & 0
    \end{tikzcd}.\]
\end{thm}

We also use the language of open book decompositions. Given a $B$-manifold $V$ with boundary and a $B$-diffeomorphism $h:V\to V$ that fixes the boundary, we can form the mapping torus $V_h$. This will have boundary $\de V \times S^1$, and we get a $B$-manifold without boundary by gluing on a copy of $\de V\times D^2$. This is called the \textit{open book} obtained from $V$ and $h$. We say a $B$-manifold $M$ has a $B$-\textit{open book decomposition} if we can pick $V$ and $h$ such that $M$ is $B$-diffeomorphic to the open book obtained from $V$ and $h$. Let us follow the convention that an open book decomposition, without specified structure, means an oriented open book decomposition. We have the following result.

\begin{thm}[\cite{skbook}, 6.4]
    If a closed connected oriented manifold $M$ has an OBD, then for any map $f:M\to X$, $[M,f] = 0$ in $\bSK_*(X)$.
\end{thm}

Note that every fiber bundle over $S^1$ is in particular an open book decomposition. On the other hand, every class in $\bSK_*(X)$ represented by a manifold with OBD is zero, hence also has a representative that fibers over the circle. So we may replace $F_n(X)$ in \cref{skbook1.2} by the subgroup of classes that have a representative with an open book decomposition.

The final theorem in this section is Theorem 1 from \cite{neumann1975}. This theorem will be generalized by our \cref{forget-Y}.

\begin{thm}
\label{forget-X}
    If $X$ is simply connected, then $\SK_n(X) \cong \SK_n$, where the isomorphism is given by $\varepsilon([M,f]) = [M]$.
\end{thm}

\section{The unoriented case} \label{section-unori}
We now move to proving the claims made in \cref{results}. In the unoriented case, we make use of the fact that there is a K\"unneth isomorphism for the unoriented bordism groups.
\begin{prop}
\label{kunneth-O}
    The map \[\Omega_*^O(X)\otimes_{\Omega_*^O}\Omega_*^O(Y)\to \Omega_*^O(X\times Y)\]
    given by product of singular manifolds is an isomorphism.
\end{prop}
\begin{proof}
    This can be shown using the fact that there is an isomorphism $\Omega_*^O(X)\cong H_*(X,\Z/2)\otimes_{\Z/2}\Omega_*^O$ (see for example \cite{conner1962differentiable}) and the usual K\"unneth formula.
\end{proof}

The claim made in \cref{komiya-unori} now follows by applying the following proposition to $Y=BO(m-n)$. Indeed $\SK_*^O(BO(k))\to\SK_*^O$ is an isomorphism. This is because generators for the homology of $BO(k)$ come from maps $\RP^i\to BO$ that classify the tautological line bundle over $\RP^i$. However, in $\SK_*^O(X)$ it is true that $[\RP^i,f]=[\RP^i,*]$ for all maps $f:\RP^i\to X$. For more details, see \cite[Thm 2.11]{skbook}.

\begin{prop}
\label{remove-BO}
    Let $Y$ be a pointed path connected space such that the forgetful map $\SK_*^O(Y)\to\SK_*^O$ is an isomorphism. Then $p_*:\SK_*^O(X\times Y)\to\SK_*^O(X)$ induced by the projection $p:X\times Y\to X$ is an isomorphism.
\end{prop}
\begin{proof}
    This map $p_*$ is trivially surjective. Let $M$ be a manifold, $f:M\to X\times Y$ a map and consider $[M,f]-[M,pf]\in\Omega_*^O(X\times Y)$, where we use the basepoint of $Y$ to embed $\Omega_*^O(X)$ into $\Omega_*^O(X\times Y)$. By \cref{kunneth-O}, we can write $[M,f]$ as a sum of elements of the form $[M_1\times M_2,f_1\times f_2]$, with $f_1:M_1\to X$, $f_2:M_2\to Y$. Then clearly $[M,pf]$ can be written as the sum of the elements $[M_1\times M_2,f_1\times *]$. We have that $[M,f]-[M,pf]$ is a sum of the elements
    \[[M_1\times M_2,f_1\times f_2] - [M_1\times M_2,f_1\times *] = [M_1,f_1]([M_2,f_2]-[M_2,*]).\]
    But recall that $[M_2,f_2]-[M_2,*]$ is zero in SK because $\SK_*^O(Y)\cong \SK_*^O$. So $[M,f]-[M,pf]$ is zero in $\bSK_*^O(X\times Y)$, and thus $p_*:\bSK_*^O(X\times Y)\to\bSK_*^O(X)$ is injective since the map $\bSK_*^O(X)\to\bSK_*^O(X\times Y)$ given by including $X$ into $X\times Y$ using the basepoint is a left inverse. Therefore it is an isomorphism. A simple application of the five lemma now shows that $p_*:\SK_*^O(X\times Y)\to \SK_*^O(X)$ is an isomorphism, as claimed.
\end{proof}

\begin{rmk}
    To the author's knowledge, there is no example in the literature that shows that unoriented SK of maps is nontrivial, i.e. a space $X$ such that $\SK^O_*(X)\not\cong\SK^O_*$. Neumann \cite[Theorem 2]{neumann1975} gives a way to compute $\SK_2(X)$ using the fundamental group of $X$ and its abelian subgroups. This is then applied with $X$ an oriented surface of genus $\geq 2$ to show that $\SK_*(X)\not\cong\SK_*$. His methods also work in the unoriented case, replacing $\Z$ coefficients with $\Z/2$ ones, with the caveat that there is now two surfaces that fiber over the circle instead of just one: apart from the torus we also have to consider the Klein bottle $K$. We therefore also need to consider subgroups of $\pi_1(X)$ that are images of $\pi_1(K)$, i.e. the free group on two generators modulo the relation $ab = b^{-1}a$. Fortunately, $\pi_1(X)$ admits no noncyclic such subgroups. Therefore $\SK_*^O(X)$ is also nontrivial.
\end{rmk}

\section{The oriented case} \label{section-ori}

The situation in the oriented case is more complicated, and we shall require powerful theory. However, we will eventually be able to prove

\begin{thm}
\label{forget-Y}
    Let $Y$ be a simply connected pointed space, and suppose $\pi_1(X)$ is finitely presented. Then the projection $p_*:\SK_n(X\times Y)\to\SK_n(X)$ is an isomorphism for all $n$.
\end{thm}

The proof given here is inspired by Neumann's proof of \cref{forget-X}. The low-dimensional cases can be easily handled here. In odd dimensions all the $\SK$-groups vanish, so there is nothing to prove. In dimension 2, Neumann \cite[Thm 2]{neumann1975} proves that the group $\SK_2(X)$ depends only on $\pi_1(X)$, and thus $\SK_2(X\times Y)\to\SK_2(X)$ must be an isomorphism. The case $n \geq 6$ is done by studying the asymmetric signature of Ranicki \cite{ranicki2013high} and concluding that this does not depend on the map to $Y$. This generalizes Neumann's argument which relied on the Open Book Decomposition Theorem of Winkelnkemper \cite{winkelnkemper-obd}. Finally, in the case $n=4$ Neumann uses the classification of simply connected 4-manifolds up to stable diffeomorphism. We need to consider non-simply connected manifolds and use the theory due to Kreck \cite{kreck1999surgery}.

\begin{proof}
    of the high-dimensional case $n\geq 5$. We need to consider the asymmetric signature, which is an invariant of manifolds that have an open book decomposition on the boundary. Let us consider the bounded book bordism groups $BB_n(X)$ of Ranicki. The group $BB_n(X)$ consists of $(n+2)$-manifolds with boundary that have a book decomposition on the boundary together with a map to $X$, modulo bordisms of manifolds with boundary that extend the book decomposition on the boundary and the map to $X$. In particular, this contains all the closed singular manifolds $[M,f]$. The class of a manifold with boundary is zero when the book decomposition on the boundary extends to the entire manifold. 

    Ranicki now defines the asymmetric signature by constructing some asymmetric Poincar\'e complex of $\Z[\pi_1(M)]$-modules from the $(n+2)$-manifold with boundary $M$ and taking its cobordism class in the asymmetric $L$-group $\operatorname{LAsy}_h^{n+2}(\Z[\pi_1(M)])$. We extend this to an invariant of $BB_n(X)$ by pulling back the coefficients along the map $f:M\to X$ to get a Poincar\'e complex of $\Z[\pi_1(X)]$-modules, so that we get a map $\sigma*: BB_n(X) \to \operatorname{LAsy}_h^{n+2}(\Z[\pi_1(X)])$. This is a bounded book bordism invariant \cite[Prop 29.16]{ranicki2013high}, so if the book decomposition on the boundary extends, then $\sigma^*(M,f) = 0$. The crucial point is that the asymmetric signature is the only obstruction, that is, when $\sigma^*(M,f) = 0$, the decomposition on the boundary extends [loc. cit., Thm 29.17], when $n\geq 3$, that is, the dimension of $M$ is at least 5.

    Consider $[M,f]$ that is zero in $\bSK_{n+2}(X)$ and let $g:M\to Y$ be a map, recalling that $Y$ is simply connected. The group $\operatorname{LAsy}_h^{n+2}(\Z[\pi])$ depends only on $\pi$ so $\operatorname{LAsy}_h^{n+2}(\Z[\pi_1(X\times Y])\cong \operatorname{LAsy}_h^{n+2}(\Z[\pi_1(X)])$. Now computing the asymmetric signature $\sigma^*(M,f\times g)$ does not depend on the map $g$, because pulling back the coefficients to $\Z(\pi_1(X\times Y)$ along $f\times g$ is the same thing as pulling back the coefficients to $\Z(\pi_1(X))$ along $f$. Therefore we have that $\sigma^*(M,f\times g)\in \operatorname{LAsy}_h^{n+2}(\Z[\pi_1(X)])$ equals $\sigma^*(M,f) = 0$ because $[M,f]$ is bordant to a manifold that has an open book decomposition. Hence the map $p_*:\bSK_n(X\times Y)\to\bSK_n(X)$ is injective. It's trivially surjective and we use the five lemma to conclude.
\end{proof}

To prove the claim for $n=4$, we need to introduce the stable diffeomorphism classification that follows from Kreck modified surgery. Let $M$ be an oriented manifold (in our case, it will always be a 4-manifold). A \textit{normal 1-type} of $M$ is a 2-coconnected fibration $B\to BSO$ such that there exists a map $M\to B$ lifting the classifying map for the normal bundle of $M$ that is 2-connected. Such a lift is called a \textit{normal 1-smoothing}. It can be shown that every $M$ has a normal $1$-type and that it is unique up to vertical homotopy equivalence. We will use the following result, which tells us that asking whether two 4-manifolds are stably diffeomorphic is a bordism problem.

\begin{thm}
    Two closed 4-dimensional manifolds with the same normal 1-type \( \xi: B \to BSO \), admitting bordant normal 1-smoothings, are stably diffeomorphic.
\end{thm}
This is essentially Theorem C in \cite{kreck1999surgery}. Note that we do not require the number of connected sums with $S^2\times S^2$ on both sides to be equal, which is why we can drop the assumption on the Euler characteristic that is made there.

We need the following lemma that tells us what the normal 1-type of a manifold is. A proof can be found in \cite[Section 3]{KLPT-stable4mnflds}. A manifold $M$ is called \textit{totally nonspin} if its universal cover $\widetilde{M}$ is not spin.

\begin{lem}
\label{totally-nonspin}
    Let \( \pi \) be a finitely presented group. The normal 1-type of a totally 
    nonspin manifold with fundamental group \( \pi \) is given by  
        \[\xi: B\pi \times BSO \xrightarrow{\operatorname{pr}_2} BSO,\]
    where the map is given by the projection onto \( BSO \). A 1-smoothing is any choice of $c:M\to B\pi$ inducing an isomorphism of fundamental groups together with the obvious map $M\to BSO$.
\end{lem}

The way we will apply this is to start with $f:M\to X$ and use $f$ as part of the 1-smoothing of $M$. Consider the following construction. Start with a bordism class $[M,f]\in \Omega_4(X)$. If $[\alpha]\in\pi_1(X)$ is a class outside the image of $f_*$, then we have a map $\alpha\circ p_1: S^1\times S^3 \to X$ whose induced map on fundamental groups does hit $[\alpha]$. Since $\pi_(X)$ is finitely presented, we can take connected sums with finitely many $[S^1\times S^3,\alpha\circ p_1]$ to get a $[M',f']$ with $f''_*$ surjective. Then we can apply surgery below the middle dimension to kill classes in $\pi_1(M')$ that are sent to zero. We again only need to do this finitely many times, so the trace of the surgeries will give a bordism to $[M'',f'']$ such that $f''_*$ is an isomorphism.

Using this, we will now prove our theorem for 4-manifolds. Let $[M,f\times h]$ be in $\SK_4(X\times Y)$ such that $[M,f] = 0$ in $\SK_4(X)$. Then we know that $[M,f]$ is bordant in $X$ to $[M',f']$ with $f'_*:\pi_1(M')\to\pi_1(X)$ an isomorphism. We also know that $[M,f]$ is bordant in $X$ to $[N,g]$, where $N$ has an open book decomposition, since $[M,f]=0\in\bSK_4(X)$. Now, we can again apply the previous lemma to get $[N',g']$ bordant to $[N,g]$ with $g'_*$ an isomorphism. Now note that the twisted $S^2$-bundle over $S^2$, denoted here by $T$, is a simply connected totally nonspin 4-manifold that is zero in $\bSK_4$. Indeed we have that $T$ fibers over $S^2$, so $[T] = [S^2][S^2] = 0$ there by \cite[Lemma 1.5]{skbook}. Now we can assume that $N'$ and $M'$ are totally nonspin by taking a connected sum with $T$ if necessary.

Because of this, we have that $f'$ gives a choice of normal 1-smoothing on $M'$ by composing with the obvious map $X\to K(\pi_1(X),1)$, and similarly for $h'$ on $N'$. Furthermore, these normal 1-smoothings are now bordant, by construction. Therefore $M'$ and $N'$ are stably diffeomorphic. Let $\varphi M'\#_a S \to N'\#_b S$ be the diffeomorphism, $a,b\geq 0$, where $S:=S^2\times S^2$. We have that the diffeomorphism commutes with the normal 1-smoothings on $M'\#_a S$ and $N'\#_b S$ (cf. \cite{kreck1999surgery}, Cor 4.3), which are exactly given by combining $f'$ and $g'$ with the constant map on $S$ and composing with $X\to K(\pi_1(X),1)$ . We conclude that in $\bSK_4(X\times Y)$
\begin{align*}
    [M,f\times h] &= [M',f'\times h'] = [M'\#_a S, (f'\times h')\#_a *] = [N'\#_b S, (f'\times h')\#_a *\circ \varphi^{-1}] \\
    &= [N', g''\times k] + [S,q_1] + \dots + [S,q_b],
\end{align*}
where $g'':N'\to X$, $k:N'\to Y$ and the $q_i:S\to X$ are some suitable maps. The classes $[S,q_i]$ are all zero in $\bSK$ since $S$ has an open book decomposition. Note that $g''$ agrees with $g'$ when composed with $X\to K(\pi_1(X),1)$ by construction of $\phi$. This means reversing the surgeries from $N$ to $N'$ is not obstructed by the map $g''$, because the only obstructions live in the fundamental group. That is, the class $[N',g''\times k] = [N,g'''\times k']$ in $\bSK_4(X\times Y)$, and this is zero because $N$ has an OBD. We conclude that $[M,f\times h]$ actually was zero. The proof can be completed with the same five lemma argument as before.

\bibliography{biblio}
\bibliographystyle{amsalpha}
\end{document}